\documentclass[11pt,a4paper]{article}

\usepackage{amsmath,amssymb,latexsym}
\usepackage{mathtools}
\usepackage{graphicx}
\usepackage{epsfig}
\usepackage{longtable}
\usepackage{amsthm}
\usepackage{xcolor}
\usepackage{multirow}
\usepackage{caption}
\usepackage{subcaption}
\usepackage{ upgreek }
\usepackage{longtable}
\usepackage{makecell} 
\usepackage{cancel}

\usepackage{lineno}

\newcommand{\I}{\mbox{${\rm I\!I}$}}

\newtheorem{theorem}{Theorem}
\newtheorem{lemma}[theorem]{Lemma}
\newtheorem{corollary}[theorem]{Corollary}

\newtheorem{definition}{Definition}
\newtheorem{example}{Example}
\newtheorem{remark}{Remark}

\newtheorem{condition}{Condition}

\def\XX{{\mathbf X}}
\def\YY{{\mathbf Y}}
\def\xx{{\mathbf x}}

\def\uu{{\mathbf u}}
\def\aa{{\mathbf a}}
\def\bb{{\mathbf b}}
\def\I{{\mathbb I}}

\begin{document}

\title{ \Large\bf Directional $\rho$-coefficients}
\author{Enrique de Amo$^{\rm a}$, David García-Fernández$^{\rm b,}\footnote{Corresponding author}$, Manuel Úbeda-Flores$^{\rm a}$\bigskip\\
\small{$^{\rm a}$Department of Mathematics, University of Almería, 04120 Almería, Spain}\\
\small{\texttt{edeamo@ual.es,\,\,\,mubeda@ual.es}}\\
\small{$^{\rm b}$Research Group of Theory of Copulas and Applications, University of Almería, 04120 Almería, Spain}\\
\small{\texttt{davidgfret@correo.ugr.es}}\\
}
\maketitle

\begin{abstract}
    In this paper we obtain advances for the concept of directional $\rho$-coefficients, originally defined for the trivariate case in [Nelsen, R.B., \'Ubeda-Flores, M. (2011). ``Directional dependence in multivariate distributions''. {\it Ann. Inst. Stat. Math} {\bf 64}, 677--685] by extending it to encompass arbitrary dimensions and directions in multivariate space. We provide a generalized definition and establish its fundamental properties. Moreover, we resolve a conjecture from the aforementioned work by proving a more general result applicable to any dimension, correcting 
   a result in [Garc\'ia, J.E., Gonz\'alez-L\'opez, V.A., Nelsen, R.B. (2013). ``A new index to measure positive dependence in trivariate distributions''. {\it J. Multivariate Anal.} {\bf 115}, 481--495] an erratum in the current literature. Our findings contribute to a deeper understanding of multivariate dependence and association, offering novel tools for detecting directional dependencies in high-dimensional settings. Finally, we introduce nonparametric estimators, based on ranks, for estimating directional $\rho$-coefficients from a sample.
\end{abstract}



\section{Introduction}\label{sec:copulas}

$\phantom{999}$Fisher, in his article on copulas within the {\it Encyclopedia of Statistical Science} \cite{Fisher97}, notes: ``Copulas [are] of interest to statisticians for two main reasons: First, as a way of studying scale-free measures of dependence; and secondly, as a starting point for
constructing families of bivariate distributions$\ldots$''

On the other hand, as Jodgeo \cite{Jodgeo82} points out: ``Dependence relations between random variables is one of the most studied subjects in probability and statistics.''

Dependence measures play a crucial role in multivariate statistics, allowing researchers to quantify the relationships between random variables beyond simple pairwise associations. Traditional measures, such as Spearman’s rho and Kendall’s tau, provide useful insights but may fail to detect complex dependency structures in higher dimensions.

To address this issue, the directional $\rho$-coefficients were introduced in \cite{Ne2011} as a tool to measure association in specific directions within trivariate distributions. These coefficients, based on copula theory, capture dependence structures that remain undetected by classical multivariate association measures. However, their study was limited to the three-dimensional case, leaving open the question of whether such coefficients in higher dimensions could be expressed as a linear combination in terms of $\rho$-coefficients in lower dimensions.

This paper extends the concept of directional $\rho$-coefficients to encompass arbitrary dimensions and directions within a multivariate setting. We establish a robust mathematical framework for their definition and demonstrate key properties that validate their practical application. Furthermore, we address a conjecture posed in \cite{Ne2011} by presenting a generalized proof, applicable across all dimensions, leveraging a result from \cite{Edeamo}.

The structure of this paper is as follows: Section \ref{sec:prel} presents the necessary preliminaries on copulas, measures of association and concepts of dependence. Section \ref{sec:main} introduces the generalization of directional $\rho$-coefficients to higher dimensions, we establish their key properties and explore their theoretical implications, alongside providing a more general response to the conjecture presented in \cite{Ne2011} and correcting \cite[Equation (10)]{Gar2013}. In Section \ref{sec:esti}, we present nonparametric rank-based estimators with the objective of estimating the values of the directional $\rho$-coefficients from the observed data in a sample. Potential applications are given in Section \ref{sec:appl}. Finally, Section \ref{sec:conc} is devoted to conclusions.

\section{Preliminaries}\label{sec:prel}

$\phantom{999}$In this section we will outline some relevant known results. Specifically, we will dedicate three subsections to: copulas, measures of association, and concepts of dependence.

\subsection{Copulas and Sklar's theorem}

$\phantom{999}$Let $n\geq 2$ be a natural number. An $n$-dimensional copula (for short, $n$-copula), $C$, is the restriction to $\I^n$ ($=[0,1]^n$) of a continuous $n$-dimensional distribution functions whose univariate margins are uniform on $\I$. Equivalently, we have the following definition:

\begin{definition}An $n$-copula $C$ is a function $C:\I^n\longrightarrow \I$ satisfying the following properties:
    \begin{itemize}
        \item [i)] $C(\uu)=0$ for every $\uu\in \I^n$ such that at least one coordinate of $\uu$ is equal to $0$.
        \item [ii)] $C(\uu)=u_k$ whenever all coordinates of $\uu$ are $1$ except, maybe, $u_k$.
        \item [iii)] For every $\aa=(a_1,a_2,\ldots,a_n)$ and $\bb=(b_1,b_2,\ldots,b_n)\in \I^n$, such that $a_k\leq b_k$ for all $k=1,\ldots,n$,
        $$V_C([\aa,\bb])=\sum_{{\bf c}}sgn({\bf c})C({\bf c})\geq 0,$$
         where $[\aa,\bb]$ denotes the $n$-box $\times_{i=1}^n[a_i,b_i]$ and the sum is taken over the vertices ${\bf c}=(c_1,c_2,\ldots,c_n)$ of $[\aa,\bb],$ i.e., each $c_k$ is equal to either $a_k$ or $b_k$, and $sgn({\bf c})=1$ if $c_k=a_k$ for an even number of $k$'s, and $sgn({\bf c})=-1$ otherwise.
    \end{itemize}
\end{definition}

The importance of copula theory lies in Sklar's theorem \cite{Sk59} ---for a complete proof of this result, see \cite{Ub2017}.

\begin{theorem}[Sklar]
Let $\XX=(X_1,X_2,...,X_n)$ be a random $n$-vector with joint distribution function $H$ and one-dimensional marginal distributions $F_1,F_2,...,F_n$. Then there exists an $n$-copula (uniquely determined on $\times_{i=1}^n\text{ Range } (F_i)$) such that
$$H(\xx)=C(F_1(x_1),F_2(x_2),...,F_n(x_n))$$
for all $\xx\in[-\infty,+\infty]^n$. If all the marginals $F_i$ are continuous, then the $n$-copula is unique.
\end{theorem}

For instance, some of the most known copulas are the product copula $\Pi_n( \uu)=\prod_{i=1}^nu_i$ ---which characterizes the independence of random variables and the subscript indicates dimension---, $M_n(\uu):=\min({\bf u})$ and $W_n(\uu)=\max\{\sum_{i=1}^nu_i-n+1,0\}$ for all ${\bf u}$ in $\I^n$. $W_n$ is an $n$-copula only when $n=2$; and for any $n$-copula $C$ we have $W_n\le C\le M_n$, so they are respectively known as the lower and upper Fr\'echet-Hoeffding bounds. Finally, let $\widehat{C}$ be the {\it survival n-copula} of $C$, i.e., $\widehat{C}({\bf u})=\mathbb{P}[{\bf U}>{\bf 1-u}]$, where {\bf U} is a random $n$-vector whose components are uniform $\I$ and with associated $n$-copula $C$. For further information about copulas, see \cite{Durante2016book,Ne06}.

\subsection{Measures of association}

$\phantom{999}$In the case of two variables, a measure of association quantifies both the strength and direction of their relationship. A value close to +1 indicates that high (or low) values of both variables tend to appear together, whereas a value near $-1$ suggests that large values of one variable are typically associated with small values of the other. One of the best-known measures of association for two random variables $(X,Y)$ is Spearman's rho. This measure relies solely on the copula $C$ associated with the pair $(X,Y)$. So we can denote it by $\rho(C)$, and its expression is given by 
$$\rho(C)=12\int_{\I^2}C(u,v){\rm d}u{\rm d}v-3.$$
Consequently, when analyzing their properties, we can assume that the random variables $X$ and $Y$ are uniformly distributed on $I$. For a detailed exploration of their characteristics, refer to \cite{Ne06} and the sources cited therein.

In this paper, we will focus on an extension of $\rho(C)$ in the context of multiple variables, where the analysis becomes more complex. To simplify and illustrate the problem, we focus on the case involving three variables. A common approach to measuring association in this setting is to take the average of the three pairwise association measures, denoted by $\rho_3^*$. However, this method is often insufficient in capturing the overall dependence among all the three variables. It is well-established that pairwise independence among a set of three random variables $(X,Y,Z)$ does not, in general, imply their mutual independence (see \cite[Example 1]{Ne2011}).

\subsection{Directional dependence}

$\phantom{999}$Among the most widely recognized concepts describing potential dependence relationships between random variables are positive (respectively, negative) upper orthant dependence (PUOD) (respectively, NUOD), and positive (respectively, negative) lower orthant dependence (PLOD) (respectively, NLOD). The concepts of PUOD and NUOD intuitively describe how a set of random variables $X_1, X_2,\ldots, X_n$ deviate from independence in terms of their likelihood of attaining large values simultaneously. Specifically, under PUOD, these variables are more prone to exhibit high values together compared to a collection of independent random variables with the same marginal distributions, whereas under NUOD, this tendency is weaker (for some results about positive dependece properties using copulas see \cite{Joe1997,Mu06,Ne06,Wei2014}). However, dependence among random variables can be characterized in other ways. For instance, considering a subset $J\subset I=\{1,2,\ldots,n\}$ one might observe a dependence structure where large (or small) values of the variables $\{X_j:j\in J\}$ are systematically associated with small (or large) values of the remaining variables $\{X_j: j\in I\setminus J\}$. This idea was explored in \cite{Que2012}, leading to the following definition.

\begin{definition}[PD$(\alpha)$ dependence concept]
    Let $\XX$ be a continuous random vector with joint distribution function $H$, and let $\alpha=(\alpha_1,\alpha_2,\ldots,\alpha_n)\in\mathbb{R}^n$ such that $|\alpha_i|=1$ for all $i=1,2,\ldots,n$. We say that $\XX$ (or $H$) is (orthant) positive (respectively, negative) dependent according to the direction $\alpha$ ---denoted by PD$(\alpha)$ (respectively, ND$(\alpha)$)--- if, for every $\xx\in[-\infty,+\infty]^n$,
\begin{equation}\label{eq:PDalpha}
\mathbb{P}\left[\bigcap_{i=1}^n(\alpha_i X_i>x_i)\right]\geq \prod_{i=1}^n\mathbb{P}[\alpha_i X_i> x_i]
\end{equation}
(respectively, by reversing the sense of the inequality ``$\ge$'' in \eqref{eq:PDalpha}).
\end{definition}

Following from the definition above, the concept PD$({\bf 1})$ (respectively, PD$({\bf -1})$), where ${\bf 1}=(1,1,\ldots,1)$ corresponds to the PUOD (respectively, PLOD); an similarly for the corresponding negative concepts. For more details, see \cite{Que2012,Que2024}.

Based on the PD$(\alpha)$ dependence concept, a dependence order is provided  in \cite{Edeamo2024} with the aim of comparing the strength of the positive dependence in a particular direction of two random vectors.

\begin{definition}[PD$(\alpha)$ order]\label{def:PDorder}
Let $\XX$ and $\YY$ be two random $n$-vectors with respective distribution functions $F$ and $G$. Let $\alpha\in\mathbb{R}^n$ such that $|\alpha_i|=1$ for all $i=1,2,\ldots,n$. $\XX$ is said to be {\it smaller than} $\YY$ {\it in the positive dependence order according to the direction} $\alpha$, denoted by $\XX\leq_{{\rm PD}(\alpha)}\YY$, if, for every $\xx\in\mathbb{R}^n$, we have 
$$\mathbb{P}\left[\bigcap_{i=1}^n(\alpha_iX_i>x_i)\right]\leq\mathbb{P}\left[\bigcap_{i=1}^n(\alpha_iY_i>x_i)\right].$$
\end{definition}

\section{Multivariate directional $\rho$-coefficients}\label{sec:main}

The directional $\rho$-coefficients were first introduced in \cite{Ne2011} order to measure directional dependence for the case $n=3$, where $(X,Y,Z)$ are three random variables with associated copula $C$. For each direction $\alpha\in\mathbb{R}^3$ such that $|\alpha_i|=1$ for $i=1,2,3$, they were defined by
$$\rho_3^{(\alpha_1,\alpha_2,\alpha_3)}(C)=8\int_{\I^3}Q_{(\alpha_1,\alpha_2,\alpha_3)}(u,v,w){\rm d}u{\rm d}v{\rm d}w,$$
where $$Q_{(\alpha_1,\alpha_2,\alpha_3)}(u,v,w)=\mathbb{P}[\alpha_1X>\alpha_1u, \alpha_2Y>\alpha_2v,\alpha_3Z>\alpha_3w]-\mathbb{P}[\alpha_1X>\alpha_1u]\mathbb{P}[\alpha_2Y>\alpha_2v]\mathbb{P}[\alpha_3Z>\alpha_3w]$$ for all $u,v,w\in \I$.

As is noted in \cite[Remark 3]{Ne2011}, $\rho_3^{\alpha}$ is not a multivariate measure of association, in general.



From now on, we will focus on the multidimensional case. In what follows, $\XX$ will be an $n$-dimensional random vector whose marginals are uniform on $[0,1]$.

Inspired by \cite{Ne96,Ne2011,Que2012} we introduce directional $\rho$-coefficients of association as in \cite{Gar2013}.

\begin{definition}
    Let $\XX$ be an $n$-dimensional random vector with associated $n$-copula $C$. The directional $\rho$-coefficients are defined for each direction $\alpha\in\mathbb{R}^n$, with $|\alpha_i|=1$ for all $i=1,\ldots,n$, as
\begin{equation}\label{def:rhoco}\rho_n^\alpha(C)=\frac{2^n(n+1)}{2^n-(n+1)}\int_{\I^n}Q_\alpha(u_1,\ldots,u_n){\rm d}u_1\cdots {\rm d}u_n,
\end{equation}
where $$Q_\alpha(u_1,\ldots,u_n)=\mathbb{P}\left[\bigcap_{i=1}^n(\alpha_iX_i>\alpha_iu_i)\right]-\prod_{i=1}^n\mathbb{P}[\alpha_iX_i>\alpha_iu_i].$$
\end{definition}

It should be noted that the directional coefficients for directions $\alpha={\bf -1}$ and $\alpha={\bf 1}$ were defined for the first time in \cite{Ne96} in terms of copulas as it follows 
$$\rho_n^-(C)=\frac{n+1}{2^n-(n+1)}\left(2^n\int_{\I^n}C(\uu){\rm d}\uu-1\right)\hspace{0.5cm}\text{and}\hspace{0.5cm}\rho_n^+(C)=\frac{n+1}{2^n-(n+1)}\left(2^n\int_{\I^n}\Pi_n(\uu){\rm d}C(\uu)-1\right),$$
respectively.

We can rewrite the probabilities using indicator functions and expectations. Let $\mathbf{1}(A)$ denote the indicator function of event $A$. Applying this, we get:
$$\rho_n^\alpha(C)=\frac{2^n(n+1)}{2^n-(n+1)}\int_{\mathbb{I}^n}\left(\mathbb{E}\left[\prod_{i=1}^n \mathbf{1}(\alpha_iX_i > \alpha_ix_i)\right]-\prod_{i=1}^n \mathbb{E}[\mathbf{1}(\alpha_iX_i > \alpha_ix_i)]\right){\rm d}x_1\cdots {\rm d}x_n$$
Here, the expectation $\mathbb{E}$ is taken with respect to the joint distribution of the random vector ${\bf X}$ under the copula $C$. For uniform marginals, $\mathbb{E}[\mathbf{1}(\alpha_iX_i > \alpha_ix_i)] = 1/2$. Thus, the formula can be written as:
$$\rho_n^\alpha(C)=\frac{2^n(n+1)}{2^n-(n+1)}\left( \mathbb{E}\left[\prod_{i=1}^n \mathbf{1}(\alpha_iX_i > \alpha_iU_i)\right] - \frac{1}{2^n} \right).$$

Some of the main properties of these coefficients are listed in the following result.

\begin{corollary}
    Let $\XX$ be an $n$-dimensional random vector with associated $n$-copula $C$, and let $\alpha\in\mathbb{R}^n$ such that $\alpha_i\in\{-1,1\}$ for all $i=1,\ldots,n$. Then:
    \begin{itemize}
        \item [i)] $\sum_{\alpha}\rho_n^\alpha(C)=0$.  
        \item [ii)] $\rho_n^\alpha(C)=\rho_n^{-\alpha}(\widehat{C})$.
    \end{itemize}
\end{corollary}

\begin{proof}
    Part $i)$ is straightforward since we have $\sum_{\alpha}\mathbb{P}[\bigcap_{i=1}^n(\alpha_i X_i>\alpha_i x_i)]=\sum_{\alpha}\prod_{i=1}^n\mathbb{P}[\alpha_i X_i>\alpha_ix_i]=1$.
    
    For part $ii)$, we know that $\widehat{C}(\uu)=\mathbb{P}[U>1-u]=\mathbb{P}[-U<u-1]=\mathbb{P}[-U>-u]$ since $U$ is uniform in $[0,1]$. Therefore,
    \begin{align*}
        \rho_n^\alpha(C)&=\frac{2^n(n+1)}{2^n-(n+1)}\int_{\mathbb{I}^n}\left(\mathbb{P}\left[\bigcap_{i=1}^n(\alpha_iX_i>\alpha_ix_i)\right]-\prod_{i=1}^n\mathbb{P}[\alpha_iX_i>\alpha_ix_i]\right){\rm d}x_1\cdots {\rm d}x_n\\
        &=\frac{2^n(n+1)}{2^n-(n+1)}\int_{\mathbb{I}^n}\left(\mathbb{P}\left[\bigcap_{i=1}^n(-\alpha_i(-X_i)>-\alpha_i(-x_i))\right]-\prod_{i=1}^n\mathbb{P}[-\alpha_i(-X_i)>-\alpha_i(-x_i)]\right){\rm d}x_1\cdots {\rm d}x_n\\
        &=\rho_n^{-\alpha}(\widehat{C}),
    \end{align*}
    as desired.
\end{proof}

Next, we provide some examples of multivariate $\rho$-coefficients for some well-known families of $n$-copulas.

\begin{example}[Directional $\rho$-coefficients for $\Pi_n$]  If $C=\Pi_n$ then we have $\rho_n^\alpha(\Pi_n)=0$ since, in this case, $\mathbb{P}\left[\bigcap_{i=1}^n(\alpha_i X_i>x_i)\right]=\prod_{i=1}^n\mathbb{P}[\alpha_i X_i> x_i]$.
\end{example}

\begin{example}[Directional $\rho$-coefficients for $M_n$]
For the minimum copula $C=M_n$, we have $X_1 = X_2 =\cdots = X_n$ in distribution. Since the marginals of the copula are uniformly distributed on $[0,1]$, the associated random variables $X_i$ are also uniformly distributed on $[0,1]$; therefore, we consider
\begin{equation}\label{eq:Probab}
v_i(u_i, \alpha_i) := \mathbb{P}[\alpha_i X_i > \alpha_i u_i] = \begin{cases} 1 - u_i & \text{if } \alpha_i = 1, \\ u_i & \text{if } \alpha_i = -1. \end{cases}
\end{equation}
We consider the integral
$$I_n(\alpha) = \int_{\I^n} \left( \mathbb{P}\left[\bigcap_{i=1}^n(\alpha_iX>\alpha_iu_i)\right] - \prod_{i=1}^n v_i(u_i, \alpha_i) \right) {\rm d}u_1 \cdots {\rm d}u_n$$
for the cases where all $\alpha_i$ are the same. We study several cases depending on $k$, i.e. the number of components of $\alpha$ equal to $-1$.
\begin{enumerate}
\item $\alpha = (1, 1,\ldots, 1)$ ($k=0$).
The joint probability is $\mathbb{P}\left[\bigcap_{i=1}^n(X > u_i)\right] = 1 - \max(u_1, \ldots, u_n)$.
The product of marginal probabilities is $\prod_{i=1}^n (1 - u_i)$.
The integral becomes:
$$
I_n((1,\ldots,1)) = \int_{\I^n} \left( 1 - \max(u_1, \ldots, u_n) - \prod_{i=1}^n (1 - u_i) \right) {\rm d}u_1 \cdots {\rm d}u_n
$$
Using the known results
$$\int_{\I^n} (1 - \max(u_1, \ldots, u_n)) {\rm d}u_1 \cdots {\rm d}u_n = \frac{1}{n+1}$$
and
$$\int_{\I^n} \prod_{i=1}^n (1 - u_i) {\rm d}u_1 \cdots {\rm d}u_n = \left(\frac{1}{2}\right)^n,$$
we get:
$$
I_n((1,\ldots,1)) = \frac{1}{n+1} - \frac{1}{2^n}.
$$
Substituting into the definition of $\rho_n^\alpha(M_n)$:
$$
\rho_n^{(1,1,\ldots,1)}(M_n) = \frac{2^n(n+1)}{2^n-(n+1)} \left( \frac{1}{n+1} - \frac{1}{2^n} \right) = \frac{2^n(n+1)}{2^n-(n+1)} \frac{2^n - (n+1)}{(n+1)2^n} = 1.
$$
\item $\alpha = (-1,-1, \ldots, -1)$ ($k=n$).
The joint probability is $\mathbb{P}\left[\bigcap_{i=1}^n(X < u_i)\right] = \min(u_1, \ldots, u_n)$.
The product of marginal probabilities is $\prod_{i=1}^n u_i$.
The integral becomes:
$$
I_n((-1,\ldots,-1)) = \int_{\I^n} \left( \min(u_1, \ldots, u_n) - \prod_{i=1}^n u_i \right) {\rm d}u_1 \cdots {\rm d}u_n
$$
Using the known result
$$\int_{\I^n} \min(u_1, \ldots, u_n) {\rm d}u_1 \cdots {\rm d}u_n = \frac{1}{n+1},$$
we get:
$$
I_n((-1,\ldots,-1)) = \frac{1}{n+1} - \frac{1}{2^n}.
$$
Substituting into the definition of $\rho_n^\alpha(M_n)$:
$$
\rho_n^{(-1,-1,\ldots,-1)}(M_n) = \frac{2^n(n+1)}{2^n-(n+1)} \left( \frac{1}{n+1} - \frac{1}{2^n} \right) = \frac{2^n(n+1)}{2^n-(n+1)} \frac{2^n - (n+1)}{(n+1)2^n} = 1.
$$
\item $k$ is even ($0 < k < n$). When $k$ is even, we have $k$ variables where we are looking at the probability that $-X_i > -u_i$ (i.e., $X_i < u_i$) and $n-k$ variables where we are looking at $X_j > u_j$. Due to the perfect positive dependence ($X_i = X \sim U[0,1]$), the joint probability is
$$\mathbb{P}\left(X < \min_{i \in I} u_i \text{ and } X > \max_{j \in J} u_j\right) = \max\left(0, \min_{i \in I} u_i - \max_{j \in J} u_j\right).$$
The integral $I_n(\alpha)$ involves the integral of this joint probability minus the product of the marginal probabilities. The integration of $\max\left(0, \min_{i \in I} u_i - \max_{j \in J} u_j\right)$ over the unit hypercube yields
$$ \frac{k! (n-k)!}{(n+1)!}.$$
Subtracting the integral of the product of marginals, which is $\left(\frac{1}{2}\right)^n$, we get:
$$I_n(\alpha) = \frac{k! (n-k)!}{(n+1)!} - \frac{1}{2^n} .$$
Multiplying by the normalization factor $\frac{2^n(n+1)}{2^n-(n+1)}$, we obtain
$$\rho_n^\alpha(M_n) = \frac{2^n(n+1)}{2^n-(n+1)} \left( \frac{k! (n-k)!}{(n+1)!} - \frac{1}{2^n} \right)$$
for even $k$.
\item $k$ is odd ($0 < k < n$). When $k$ is odd, the joint probability is $$\mathbb{P}\left[\bigcap_{i=1}^n(\alpha_iX>\alpha_iu_i)\right] = \max\left(0, \min_{j \in J} u_j - \max_{i \in I} u_i\right),$$
where $|I| = k$ and $|J| = n-k$. The formula for $\rho_n^\alpha(M_n)$ is the same as the one derived for the even case of $k$. The difference in the directional dependence for odd $k$ arises from the structure of the joint probability involving $\min_{j \in J} u_j - \max_{i \in I} u_i$, which effectively reverses the roles of the minimum and maximum compared to the even $k$ case where we had $\min_{i \in I} u_i - \max_{j \in J} u_j$. However, the integral of $\max(0, a-b)$ is the same as the integral of $\max(0, b-a)$ over the appropriate domains, leading to the same formula for $I_n(\alpha)$ and consequently for $\rho_n^\alpha(M_n)$.
\end{enumerate}
In summary, where $k$ is the number of $-1$ in $\alpha$:
$$
\rho_n^\alpha(M_n) = \begin{cases}
1 & \text{if } k = 0 \text{ or } k = n, \\
\displaystyle\frac{2^n(n+1)}{2^n-(n+1)} \left( \frac{k! (n-k)!}{(n+1)!} - \frac{1}{2^n} \right) & \text{otherwise}.
\end{cases}
$$
\end{example}

\begin{example}[Directional $\rho$-coefficients for the Farlie-Gumbel-Morgenstern (FGM) family of $n$-copulas] \label{ex:FGM}Let us consider the FGM family of $n$-copulas given by
$$C_n^{\lambda}(u_1,\ldots,u_n)=\prod_{i=1}^nu_i\left(1+\lambda\prod_{i=1}^n(1-u_i)\right),\hspace{0.6cm}\lambda\in[-1,1]$$
(see \cite{Durante2016book,Ne06}). From \eqref{eq:Probab} we have
$$\mathbb{P}[\alpha_iX_i>\alpha_iu_i]=1-\frac{1+\alpha_i}{2}u_i-\frac{1-\alpha_i}{2}(1-u_i).$$
The joint probability $\mathbb{P}[\bigcap_{i=1}^n(\alpha_iX_i>\alpha_iu_i)]$ is given by the (directional) function 
\begin{align*}
    S^\alpha(u_1,\ldots,u_n):=&\prod_{i=1}^n\left(1-\frac{1+\alpha_i}{2}u_i-\frac{1-\alpha_i}{2}(1-u_i)\right)\\
    &\cdot\left(1+(-1)^{|J|}\lambda\prod_{i=1}^n\left(\frac{1+\alpha_i}{2}u_i+\frac{1-\alpha_i}{2}(1-u_i)\right)\right).
\end{align*}
Then, 
\begin{align*}
    Q_\alpha(u_1,\ldots,u_n)&=S^\alpha(u_1,\ldots,u_n)-\prod_{i=1}^n\mathbb{P}[\alpha_iX_i>\alpha_iu_i]\\
    &=(-1)^{|J|}\lambda\prod_{i=1}^n\mathbb{P}[\alpha_iX_i>\alpha_iu_i](1-\mathbb{P}[\alpha_iX_i>\alpha_iu_i]).
\end{align*}
Now, we calculate the integral of $Q_\alpha$ over $\I^n$:
$$\int_{\mathbb{I}^n}Q_\alpha(u_1,\ldots,u_n)du_1\cdots du_n=(-1)^{|J|}\lambda\prod_{i=1}^n\int_0^1\mathbb{P}[\alpha_iX_i>\alpha_iu_i](1-\mathbb{P}[\alpha_iX_i>\alpha_iu_i])du_i.$$
After some elementary calculations, we obtain that, for every $\alpha_i\in\{-1,1\}$,
$$\int_0^1\mathbb{P}[\alpha_iX_i>\alpha_iu_i](1-\mathbb{P}[\alpha_iX_i>\alpha_iu_i])du_i=\frac{1}{6}.$$
Thus, 
$$\int_{\mathbb{I}^n}Q_\alpha(u_1,\ldots,u_n)du_1\cdots du_n=(-1)^{|J|}\lambda\left(\frac{1}{6}\right)^n=(-1)^{|J|}\frac{\lambda}{6^n}.$$
Finally, it holds
$$\rho_n^\alpha(C_n^{\lambda})=\frac{(-1)^{|J|}2^n(n+1)\lambda}{[2^n-(n+1)]6^n}.$$
To illustrate a simple case, let $\alpha\in\mathbb{R}^4$  be a direction. Then, the respective directional $\rho$-coefficients are given by 
 $$\rho_4^\alpha(C_\lambda^4)=\begin{cases}
     -5\lambda/891, & \text{if } |J| \text{ is odd,}\\
     5\lambda/891, & \text{if } |J| \text{ is even,}\\
 \end{cases}$$
 where $J\subseteq\{1,2,\ldots,n\}$ is such that $i\in J$ if $\alpha_i=1$, for $i=1,2,\ldots,n$
 \end{example}

The following result can be trivially extracted from Definition \ref{def:PDorder} and  the expression of the directional $\rho$-coefficient given by \eqref{def:rhoco}.

\begin{corollary}
    Let $\XX$ and $\YY$ be two $n$-dimensional random vectors of continuous random variables uniformly distributed on $[0,1]$, whose respective joint distribution functions are the $n$-copulas $C_{\XX}$ and $C_{\YY}$. Let $\alpha=(\alpha_1,...,\alpha_n)\in\mathbb{R}^n$ such that $|\alpha_i|=1$ for all $i=1,...,n$. If $\XX\leq_{{\rm PD}(\alpha)}\YY$, then $\rho_n^\alpha(C_{\XX})\leq\rho_n^\alpha(C_{\YY})$.
\end{corollary}

For instance, if $C^{\lambda_1}_n$ and $C^{\lambda_2}_n$ are two members of the FGM family of $n$-copulas given in Example \ref{ex:FGM}, it is known that $C^{\lambda_1}_n\leq_{{\rm PD}(\alpha)}C^{\lambda_2}_n$ if, and only if, $\lambda_1\leq\lambda_2$ (\cite{Edeamo2024}). Therefore, $\rho_n^\alpha(C^{\lambda_1}_n)\leq\rho_n^\alpha(C^{\lambda_2}_n)$ if, and only if, $\lambda_1\leq\lambda_2$.

In this work, we address a conjecture by Nelsen and Úbeda-Flores (\cite{Ne2011}) which posits that the multidimensional directional $\rho$-coefficient can be expressed as a linear combination of lower-dimensional directional $\rho$-coefficients. Specifically, \cite{Ne2011} suggests this involves a linear combination of $\rho^+$ and $\rho^-$ coefficients. Building on this, and noting the complexity highlighted by \cite{Gar2013} regarding the representation of $2^n$ directional coefficients $\rho_n^{\alpha}$ for $n\geq 4$ in terms of lower-dimensional marginal $\rho$-coefficients, we prove that any directional $\rho$-coefficient can be expressed as a linear combination of $\rho^-$ coefficients of the constituent random variables ---we adopt the convention that the sum (respectively, product) over an empty set is 0 (respectively, 1)--- and correcting \cite[Equation (10)]{Gar2013}, as well.

\begin{theorem} \label{theo: rho}
     Let $\XX$ be an $n$-dimensional random vector whose marginals are uniform on $[0,1]$, let $C$ be its associated $n$-copula and $\alpha\in\mathbb{R}^n$ such that $\alpha_i\in\{-1,1\}$ for all $i=1,\ldots,n$. Let $I\subseteq\{1,2,\ldots,n\}$ such that $\alpha_i=-1$ if $i\in I$ and $\alpha_i=1$ if $i\in J=\{1,2,\ldots,n\}\setminus I$
     . Then
 $$\rho_n^\alpha(C)=\frac{2^n(n+1)}{2^n-(n+1)} \sum_{S \subseteq J
 } (-1)^{|S|} \frac{2^{|I|+|S|}-(|I|+|S|+1)}{2^{|I|+|S|}(|I|+|S|+1)} \rho_{{\bf X}_{I \cup S}}^-$$
     where ${\bf X}_K$ is the random vector whose components are $X_i$ for $i\in K$, and $\rho_{{\bf X}_{K}}^-$ corresponds to the coefficient $\rho_k^-$ for ${\bf X}_K$.
\end{theorem}

\begin{proof} By using \cite[Lemma 3.4 and Remark 3.5]{Edeamo}, we have
     \begin{align*}
        \rho_n^\alpha(C)&=\frac{2^n(n+1)}{2^n-(n+1)}\int_{\mathbb{I}^n}\left(\mathbb{P}\left[\bigcap_{i=1}^n(\alpha_iX_i>\alpha_ix_i)\right]-\prod_{i=1}^n\mathbb{P}[\alpha_iX_i>\alpha_ix_i]\right){\rm d}x_1\cdots {\rm d}x_n\\
        &=\frac{2^n(n+1)}{2^n-(n+1)}\int_{\mathbb{I}^n}\left(\mathbb{P}\left[\bigcap_{i\in I}X_i<x_i,\bigcap_{i\in J}X_i>x_i\right]-\prod_{i\in I}\mathbb{P}[X_i<x_i]\prod_{i\in J}\mathbb{P}[X_i>x_i]\right){\rm d}x_1\cdots {\rm d}x_n\\
        &=\frac{2^n(n+1)}{2^n-(n+1)}\int_{\mathbb{I}^n}\left( C(\xx_I,{\bf 1})-\sum_{j\in J}C(\xx_{I\cup\{j\}},{\bf1})+\sum_{j_1\in J}\sum_{\substack{j_2\in J\\ j_2>j_1}}C(\xx_{I\cup\{j_1,j_2\}},{\bf 1})-\cdots+(-1)^{|J|}C(\xx)\right.\\
        &\left.\hspace{0.8cm}-\prod_{i\in I}x_i\prod_{i\in J}(1-x_i)\right){\rm d}x_1\cdots {\rm d}x_n\\
        &=\frac{2^n(n+1)}{2^n-(n+1)}\int_{\mathbb{I}^n}\left( \sum_{k=0}^{|J|} (-1)^k \sum_{\substack{S \subseteq J \\ |S|=k}} C(\mathbf{x}_{I \cup S}, \mathbf{1}) - \prod_{i\in I}x_i\prod_{i\in J}(1-x_i) \right){\rm d}{\bf x}.\\
    \end{align*}
    Note 
\begin{align*}
\int_{\mathbb{I}^n} C(\mathbf{x}_{I \cup S},{\bf 1}) {\rm d}\mathbf{x} &= \int_{\mathbb{I}^n} C(\mathbf{x}_{I \cup S},{\bf 1}) {\rm d}\mathbf{x} - \frac{1}{2^{|I|+|S|}} + \frac{1}{2^{|I|+|S|}} \\
&= \int_{\mathbb{I}^n} \left(C(\mathbf{x}_{I \cup S},{\bf 1}) - \prod_{i \in I \cup S} x_i \right) {\rm d}\mathbf{x} + \frac{1}{2^{|I|+|S|}} \\
&= \frac{\frac{2^{|I|+|S|}(|I|+|S|+1)}{2^{|I|+|S|}-(|I|+|S|+1)} \int_{\mathbb{I}^n} \left(C(\mathbf{x}_{I \cup S},{\bf 1}) - \prod_{i \in I \cup S} x_i \right) {\rm d}\mathbf{x}}{\frac{2^{|I|+|S|}(|I|+|S|+1)}{2^{|I|+|S|}-(|I|+|S|+1)}} + \frac{1}{2^{|I|+|S|}} \\
&= \frac{2^{|I|+|S|}-(|I|+|S|+1)}{2^{|I|+|S|}(|I|+|S|+1)} \rho_{\mathbf{X}_{I \cup S}}^- + \frac{1}{2^{|I|+|S|}}
\end{align*}
and    
$$\int_{\mathbb{I}^n}\prod_{i\in I}x_i\prod_{i\in J}(1-x_i){\rm d}x_1\cdots {\rm d}x_n=\frac{1}{2^n}.$$
Considering the alternating sign and the number of such subsets $\binom{|J|}{k}$, the expression for $\rho_n^\alpha(C)$ becomes:
\begin{align*}
\rho_n^\alpha(C)=&\frac{2^n(n+1)}{2^n-(n+1)}\left( \sum_{k=0}^{|J|} (-1)^k \sum_{\substack{S \subseteq J \\ |S|=k}} \left[ \frac{2^{|I|+k}-(|I|+k+1)}{2^{|I|+k}(|I|+k+1)}\rho_{\mathbf{X}_{I \cup S}}^- + \frac{1}{2^{|I|+k}} \right] - \frac{1}{2^n} \right) \\
=&\frac{2^n(n+1)}{2^n-(n+1)}\left( \sum_{S \subseteq J} (-1)^{|S|} \frac{2^{|I|+|S|}-(|I|+|S|+1)}{2^{|I|+|S|}(|I|+|S|+1)}\rho_{\mathbf{X}_{I \cup S}}^- + \sum_{k=0}^{|J|} (-1)^k \binom{|J|}{k} \frac{1}{2^{|I|+k}} - \frac{1}{2^n} \right).
\end{align*}
The term that does not involve the $\rho^-$ coefficients is:
\begin{align*}
\sum_{k=0}^{|J|} (-1)^k \binom{|J|}{k} \frac{1}{2^{|I|+k}} - \frac{1}{2^n} &= \frac{1}{2^{|I|}} \sum_{k=0}^{|J|} \binom{|J|}{k} \left(-\frac{1}{2}\right)^k - \frac{1}{2^n} \\
&= \frac{1}{2^{|I|}} \left(1 - \frac{1}{2}\right)^{|J|} - \frac{1}{2^n} \\
&= \frac{1}{2^{|I|}} \left(\frac{1}{2}\right)^{|J|} - \frac{1}{2^{|I|+|J|}} = \frac{1}{2^{|I|+|J|}} - \frac{1}{2^n} = 0.
\end{align*}
Therefore, the desired formula is obtained.
\end{proof}



From Theorem \ref{theo: rho}, it can be verified that the expressions for $n=3$ align with those presented in \cite[Theorem 3.1]{Gar2013}. Furthermore, Table \ref{tab:n4_coefficients} provides the sixteen $\rho$-coefficients for $n=4$.
\begin{table}[htb]
\begin{center}
\begin{longtable}{c p{0.8\textwidth}}
\hline
\textbf{Direction $\alpha$} & \textbf{Coefficient $\rho_4^\alpha(C)$} \\
\hline
\endhead 
\hline
\multicolumn{2}{|r|}{{Continued on next page}} \\
\endfoot 
\hline
\endlastfoot 
$(1,1,1,1)$ & $\frac{20}{33}(\rho^-_{\mathbf{X}_{12}} + \rho^-_{\mathbf{X}_{13}} + \rho^-_{\mathbf{X}_{14}} + \rho^-_{\mathbf{X}_{23}} + \rho^-_{\mathbf{X}_{24}} + \rho^-_{\mathbf{X}_{34}}) - \frac{10}{11}(\rho^-_{\mathbf{X}_{123}} + \rho^-_{\mathbf{X}_{124}} + \rho^-_{\mathbf{X}_{134}} + \rho^-_{\mathbf{X}_{234}}) + \rho^-_{\mathbf{X}_{1234}}$ \\

$(-1,1,1,1)$ & $-\frac{20}{33}(\rho^-_{\mathbf{X}_{12}} + \rho^-_{\mathbf{X}_{13}} + \rho^-_{\mathbf{X}_{14}}) + \frac{10}{11}(\rho^-_{\mathbf{X}_{123}} + \rho^-_{\mathbf{X}_{124}} + \rho^-_{\mathbf{X}_{134}}) - \rho^-_{\mathbf{X}_{1234}}$ \\

$(1,-1,1,1)$ & $-\frac{20}{33}(\rho^-_{\mathbf{X}_{12}} + \rho^-_{\mathbf{X}_{23}} + \rho^-_{\mathbf{X}_{24}}) + \frac{10}{11}(\rho^-_{\mathbf{X}_{123}} + \rho^-_{\mathbf{X}_{124}} + \rho^-_{\mathbf{X}_{234}}) - \rho^-_{\mathbf{X}_{1234}}$ \\

$(1,1,-1,1)$ & $-\frac{20}{33}(\rho^-_{\mathbf{X}_{13}} + \rho^-_{\mathbf{X}_{23}} + \rho^-_{\mathbf{X}_{34}}) + \frac{10}{11}(\rho^-_{\mathbf{X}_{123}} + \rho^-_{\mathbf{X}_{134}} + \rho^-_{\mathbf{X}_{234}}) - \rho^-_{\mathbf{X}_{1234}}$ \\

$(1,1,1,-1)$ & $-\frac{20}{33}(\rho^-_{\mathbf{X}_{14}} + \rho^-_{\mathbf{X}_{24}} + \rho^-_{\mathbf{X}_{34}}) + \frac{10}{11}(\rho^-_{\mathbf{X}_{124}} + \rho^-_{\mathbf{X}_{134}} + \rho^-_{\mathbf{X}_{234}}) - \rho^-_{\mathbf{X}_{1234}}$ \\

$(-1,-1,1,1)$ & $\frac{20}{33}\rho^-_{\mathbf{X}_{12}} - \frac{10}{11}(\rho^-_{\mathbf{X}_{123}} + \rho^-_{\mathbf{X}_{124}}) + \rho^-_{\mathbf{X}_{1234}}$ \\

$(-1,1,-1,1)$ & $\frac{20}{33}\rho^-_{\mathbf{X}_{13}} - \frac{10}{11}(\rho^-_{\mathbf{X}_{123}} + \rho^-_{\mathbf{X}_{134}}) + \rho^-_{\mathbf{X}_{1234}}$ \\

$(-1,1,1,-1)$ & $\frac{20}{33}\rho^-_{\mathbf{X}_{14}} - \frac{10}{11}(\rho^-_{\mathbf{X}_{124}} + \rho^-_{\mathbf{X}_{134}}) + \rho^-_{\mathbf{X}_{1234}}$ \\

$(1,-1,-1,1)$ & $\frac{20}{33}\rho^-_{\mathbf{X}_{23}} - \frac{10}{11}(\rho^-_{\mathbf{X}_{123}} + \rho^-_{\mathbf{X}_{234}}) + \rho^-_{\mathbf{X}_{1234}}$ \\

$(1,-1,1,-1)$ & $\frac{20}{33}\rho^-_{\mathbf{X}_{24}} - \frac{10}{11}(\rho^-_{\mathbf{X}_{124}} + \rho^-_{\mathbf{X}_{234}}) + \rho^-_{\mathbf{X}_{1234}}$ \\

$(1,1,-1,-1)$ & $\frac{20}{33}\rho^-_{\mathbf{X}_{34}} - \frac{10}{11}(\rho^-_{\mathbf{X}_{134}} + \rho^-_{\mathbf{X}_{234}}) + \rho^-_{\mathbf{X}_{1234}}$ \\

$(-1,-1,-1,1)$ & $\frac{10}{11}\rho^-_{\mathbf{X}_{123}} - \rho^-_{\mathbf{X}_{1234}}$ \\

$(-1,-1,1,-1)$ & $\frac{10}{11}\rho^-_{\mathbf{X}_{124}} - \rho^-_{\mathbf{X}_{1234}}$ \\

$(-1,1,-1,-1)$ & $\frac{10}{11}\rho^-_{\mathbf{X}_{134}} - \rho^-_{\mathbf{X}_{1234}}$ \\

$(1,-1,-1,-1)$ & $\frac{10}{11}\rho^-_{\mathbf{X}_{234}} - \rho^-_{\mathbf{X}_{1234}}$ \\

$(-1,-1,-1,-1)$ & $\rho^-_{\mathbf{X}_{1234}}$ \\
\hline
\end{longtable}
\caption{The sixteen directional $\rho$-coefficients for $n=4$.} \label{tab:n4_coefficients}
\end{center}
\end{table}

\section{Nonparametric estimators for directional dependence}\label{sec:esti}

There are several possible approaches to defining an estimator for the directional $\rho$-coefficients since, in practice, the copula $C$ remains unknown and must be estimated from the available data. Perhaps the most immediate way to define it could be using the empirical copula; however, in \cite{Pe2016}, Pérez and Prieto-Alariz already warn of the lack of good properties of the estimators resulting from this method; for example, these may fall outside the parametric space (\cite[Example 1]{Pe2016}). That is why we have decided that the estimators, $\widehat{\rho}_n^\alpha$, proposed in this paper are based on ranks.

Inspired by \cite{Gar2013}, where a nonparametric estimator for tridimensional directional $\rho$-coefficients is defined, we introduce a nonparametric estimator for multivariate directional $\rho$-coefficients as follows:

Consider a $d$-dimensional random sample of size $n$, $\{X_{1j},\ldots,X_{dj}\}_{j=1}^n$, of a $d$-dimensional random vector $(X_1,\ldots,X_d)$ with associated $d$-copula $C$. Let $R_{ij}$ be the rank of $X_{ij}$ in $\{X_{i1},\ldots,X_{in}\}$, $\overline{R}_{ij}=n+1-R_{ij}$ and let $R_{ij}^{\alpha_i}$ be $R_{ij}$ if $\alpha_i=-1$ and $\overline{R}_{ij}=n+1-R_{ij}$ if $\alpha_i=1$ , for $i=1,\ldots,d$. Then we define
$$\widehat{\rho}_d^{\alpha}=\frac{\frac{1}{n}\sum_{j=1}^n\prod_{i=1}^dR_{ij}^{\alpha_i}-\left(\frac{n+1}{2}\right)^d}{\frac{1}{n}\sum_{j=1}^nj^d-\left(\frac{n+1}{2}\right)^d}.$$

Note that in the case of perfect positive dependence, that is, in the case in which one random variable increases when the others also increase ---i.e., the ranks in each dimension coincide---, we have $\widehat{\rho}_d^+=\widehat{\rho}_d^-=1$. Every other directional estimator is equal to $1$ provided that the ranks $R_{ij}^{\alpha_i}$ are equal for any $i=1,\ldots,d$ and for every $j=1,\ldots,n$. It is also worth observing that the formula for $d=3$ matches the one presented in \cite {Gar2013}:
$$\widehat{\rho}_3^{(\alpha_1,\alpha_2,\alpha_3)}=\frac{8}{n(n-1)(n+1)^2}\sum_{j=1}^nR_{1j}^{\alpha_1}R_{2j}^{\alpha_2}R_{3j}^{\alpha_3}-\frac{n+1}{n-1}.$$

For this dimension, $d=3$, we can define one more coefficient, the arithmetic mean of the three two-dimensional coefficients, $\rho_3^*=(\rho_{12}+\rho_{13}+\rho_{23})/3$. To estimate this coefficient we can calculate the arithmetic mean of the estimators of the two-dimensional coefficients as in \cite{Gar2013}, i.e., $\widehat{\rho}_3^*=(\widehat{\rho}_{12}+\widehat{\rho}_{13}+\widehat{\rho}_{23})/3$. However, as with the coefficient, this estimator may not capture dependence relationships between the samples of the random variables, as we illustrate in the following example.

\begin{example}
     For the FGM family of 3-copulas with $\lambda=0.6$, we performed a Monte Carlo simulation to generate 1000 samples, each of size 500. For each of these samples, both the $\rho_3^*$ coefficient and the directional $\rho$-coefficients $\rho_3^-$ and $\rho_3^+$ have been calculated. The arithmetic mean of all the results obtained in the samples yielded the following values
     $$\widehat{\rho}_3^*=0.0006\hspace{1cm}\widehat{\rho}_3^-=-0.0215\hspace{1cm}\widehat{\rho}_3^+=0.0217.$$
     So the value of $\widehat{\rho}_3^*$ tends to $0$ as $n$ increases, therefore, there may be dependence among the samples of the random variables undetected by the estimator $\widehat{\rho}_3^*$.
\end{example}

It is possible to express these estimators as a linear combination of the estimators of the lower-dimensional $\widehat{\rho}^-$-coefficients, analogously to the combination presented in Theorem \ref{theo: rho}. In order to prove this result, we first establish the following technical lemma.

\begin{lemma}\label{lem:aux}
\label{lem:generalized_product_decomposition}
Let $U_1, \ldots, U_d$ be random variables with marginal distributions uniform on $[0,1]$. For any partition of the set of indices $\{1, \ldots, d\}$ into two disjoint sets $I$ and $J$ (i.e., $I \cup J = \{1, \ldots, d\}$ and $I \cap J = \emptyset$), the following identity holds for the expectation of a mixed product:
$$ \mathbb{E}\left[ \prod_{i \in I} U_i \prod_{i \in J} (1-U_i) \right] = \sum_{S \subseteq J} (-1)^{|S|} \mathbb{E}\left[ \prod_{l \in I \cup S} (1 - U_l) \right]. $$
This identity also applies to empirical averages for a sample $U_{1j}, \ldots, U_{dj}$ for $j=1, \ldots, n$:
$$ \frac{1}{n}\sum_{j=1}^n \left(\prod_{i \in I} U_{ij}\right) \left(\prod_{i \in J} (1-U_{ij})\right) = \sum_{S \subseteq J} (-1)^{|S|} \left( \frac{1}{n} \sum_{j=1}^n \prod_{l \in I \cup S} (1 - U_{lj}) \right). $$
\end{lemma}

\begin{proof}
Consider the product term $\prod_{i \in I} U_i \prod_{i \in J} (1-U_i)$. We can rewrite each $U_i$ for $i \in I$ as $1 - (1-U_i)$:
$$ \prod_{i \in I} U_i = \prod_{i \in I} (1 - (1-U_i)).$$
Applying the inclusion-exclusion principle for products (i.e., expanding the product of terms $(1-x_i)$ as $\sum_{S' \subseteq I} (-1)^{|S'|} \prod_{s' \in S'} x_{s'}$ with $x_i = (1-U_i)$), we get:
$$ \prod_{i \in I} (1 - (1-U_i)) = \sum_{S' \subseteq I} (-1)^{|S'|} \prod_{s' \in S'} (1-U_{s'}).$$
Substitute this back into the original product:
$$ \left( \sum_{S' \subseteq I} (-1)^{|S'|} \prod_{s' \in S'} (1-U_{s'}) \right) \left( \prod_{i \in J} (1-U_i) \right).$$
Because the distributive law, and since $I$ and $J$ are disjoint, and $S' \subseteq I$, the sets $S'$ and $J$ are also disjoint, the product of terms $(1-U_l)$ can be combined over the union of their index sets $S' \cup J$:
$$ \sum_{S' \subseteq I} (-1)^{|S'|} \left( \prod_{s' \in S'} (1-U_{s'}) \prod_{i \in J} (1-U_i) \right) = \sum_{S' \subseteq I} (-1)^{|S'|} \prod_{l \in S' \cup J} (1-U_l).$$
Now, let us take the expectation (or empirical average) of both sides.
$$ \mathbb{E}\left[ \prod_{i \in I} U_i \prod_{i \in J} (1-U_i) \right] = \mathbb{E}\left[ \sum_{S' \subseteq I} (-1)^{|S'|} \prod_{l \in S' \cup J} (1-U_l) \right];$$
By the linearity of expectation:
$$ \mathbb{E}\left[ \prod_{i \in I} U_i \prod_{i \in J} (1-U_i) \right] = \sum_{S' \subseteq I} (-1)^{|S'|} \mathbb{E}\left[ \prod_{l \in S' \cup J} (1-U_l) \right].$$
This form is equivalent to the identity in the lemma statement by simply relabeling $S'$ as $S$. The extension to empirical averages follows directly by replacing expectations with sample means.
\end{proof}

We are now ready to prove that the directional coefficients can be represented as a linear combination of $\widehat{\rho}^-$ lower-dimensional coefficients.

\begin{theorem}{\label{th:est}}
    Let $\{X_{1j},\ldots,X_{d_j}\}_{j=1}^n$ be a $d$-dimensional random sample of size $n$ of a $d$-dimensional random vector $(X_1,\ldots,X_d)$ with associated $d$-copula $C$. Let $\alpha\in\mathbb{R}^n$ be such that $\alpha_i\in\{-1,1\}$ for all $i=1,\ldots,n$. Let $I\subseteq\{1,2,\ldots,n\}$ such that $\alpha_i=-1$ if $i\in I$ and $\alpha_i=1$ if $i\in J=\{1,2,\ldots,n\}\setminus I$
    . Then
\begin{align}\label{eq:estimator}
\widehat{\rho}_d^{\alpha} = \frac{(n+1)^d}{\frac{1}{n}\sum_{j=1}^nj^d-\left(\frac{n+1}{2}\right)^d} \sum_{k=0}^{|J|} (-1)^k \frac{2^{|I|+k}-(|I|+k+1)}{2^{|I|+k}(|I|+k+1)} \sum_{\substack{S \subseteq J \\ |S|=k}} \widehat{\rho}_{X_{I\cup S}}^-
\end{align}
\end{theorem}

\begin{proof}Let $\mathbf{X}_1, \ldots, \mathbf{X}_n$ be a random sample from a $d$-variate distribution, and let $R_{ij}^{\alpha_i}$ the (directional) ranks. The pseudo-observations are defined as $U_{ij} = R_{ij}/(n+1)$ for $i=1,\ldots,d$ and $j=1,\ldots,n$. For a subset of indices $K \subseteq \{1,\ldots,d\}$ with $|K|$ components, the empirical upper-tail dependence coefficient $\widehat{\rho}_K^-$ is defined as:
$$\widehat{\rho}_K^- = \frac{2^{|K|}(|K|+1)}{2^{|K|}-(|K|+1)} \left( \frac{1}{n} \sum_{j=1}^n \prod_{i \in K} (1 - U_{ij}) - \left(\frac{1}{2}\right)^{|K|} \right).$$
Now, let $N_{\text{rank}}$ denote the numerator of $\widehat{\rho}_d^{\alpha}$ and $D_{\text{rank}}$ denote its denominator, i.e.,
$$N_{\text{rank}} = \frac{1}{n}\sum_{j=1}^n\prod_{i=1}^dR_{ij}^{\alpha_i}-\left(\frac{n+1}{2}\right)^d$$
$$D_{\text{rank}} = \frac{1}{n}\sum_{j=1}^nj^d-\left(\frac{n+1}{2}\right)^d$$
By definition, $R_{ij}^{\alpha_i} = (n+1)U_{ij}$ if $\alpha_i = -1$ (i.e., $i \in I$), and $R_{ij}^{\alpha_i} = (n+1)(1-U_{ij})$ if $\alpha_i = 1$ (i.e., $i \in J$). Thus, the product $\prod_{i=1}^d R_{ij}^{\alpha_i}$ becomes:
\begin{align*} \prod_{i=1}^d R_{ij}^{\alpha_i} &= \left(\prod_{i \in I} (n+1)U_{ij}\right) \left(\prod_{i \in J} (n+1)(1-U_{ij})\right) \\ &= (n+1)^{|I|} (n+1)^{|J|} \left(\prod_{i \in I} U_{ij}\right) \left(\prod_{i \in J} (1-U_{ij})\right) \\ &= (n+1)^d \left(\prod_{i \in I} U_{ij}\right) \left(\prod_{i \in J} (1-U_{ij})\right).
\end{align*}
Substituting this expression into the expression for $N_{\text{rank}}$:
\begin{align*} N_{\text{rank}} &= \frac{1}{n}\sum_{j=1}^n (n+1)^d \left(\prod_{i \in I} U_{ij}\right) \left(\prod_{i \in J} (1-U_{ij})\right) - \left(\frac{n+1}{2}\right)^d \\ &= (n+1)^d \left[ \frac{1}{n}\sum_{j=1}^n \left(\prod_{i \in I} U_{ij}\right) \left(\prod_{i \in J} (1-U_{ij})\right) - \left(\frac{1}{2}\right)^d \right]. \end{align*}
Let $K_{\alpha}^{\text{emp}}$ denote the term in the square brackets:
$$K_{\alpha}^{\text{emp}} = \frac{1}{n}\sum_{j=1}^n \left(\prod_{i \in I} U_{ij}\right) \left(\prod_{i \in J} (1-U_{ij})\right) - \left(\frac{1}{2}\right)^d;$$
so, $N_{\text{rank}} = (n+1)^d K_{\alpha}^{\text{emp}}$. Lemma \ref{lem:aux} states that $K_{\alpha}^{\text{emp}}$ can be expressed as:
$$K_{\alpha}^{\text{emp}} = \sum_{S \subseteq J} (-1)^{|S|} \left( \frac{1}{n} \sum_{j=1}^n \prod_{l \in I \cup S} (1 - U_{lj}) - \left(\frac{1}{2}\right)^{|I|+|S|} \right).$$
This identity allows us to decompose $K_{\alpha}^{\text{emp}}$ into a sum of terms related to products of $(1-U_{lj})$.
From the definition of $\widehat{\rho}_K^-$, we can isolate the parenthesized term:
$$ \frac{1}{n} \sum_{j=1}^n \prod_{i \in K} (1 - U_{ij}) - \left(\frac{1}{2}\right)^{|K|} = \frac{2^{|K|}-(|K|+1)}{2^{|K|}(|K|+1)} \widehat{\rho}_K^-. $$
Substituting this, and by setting $K = I \cup S$, we have:
$$K_{\alpha}^{\text{emp}} = \sum_{S \subseteq J} (-1)^{|S|} \left( \frac{2^{|I|+|S|}-(|I|+|S|+1)}{2^{|I|+|S|}(|I|+|S|+1)} \widehat{\rho}_{X_{I\cup S}}^- \right).$$
This sum can be regrouped by the size of $S$, $k=|S|$:
$$K_{\alpha}^{\text{emp}} = \sum_{k=0}^{|J|} (-1)^k \frac{2^{|I|+k}-(|I|+k+1)}{2^{|I|+k}(|I|+k+1)} \sum_{\substack{S \subseteq J \\ |S|=k}} \widehat{\rho}_{X_{I\cup S}}^-.$$
Finally, recall that $\widehat{\rho}_d^{\alpha} = \frac{N_{\text{rank}}}{D_{\text{rank}}}$ and substitute the expression for $N_{\text{rank}} = (n+1)^d K_{\alpha}^{\text{emp}}$, whence Equation \eqref{eq:estimator} follows, completing the proof.
\end{proof}

\begin{remark}Observe that, since
\begin{align*}
\lim_{n \to \infty} \frac{(n+1)^d}{\frac{1}{n}\sum_{l=1}^nl^d-\left(\frac{n+1}{2}\right)^d}=& \lim_{n \to \infty} \frac{n^d\left(1+\frac{1}{n}\right)^d}{\left(\frac{n^d}{d+1} + O(n^{d-1})\right)-\left(\frac{n^d}{2^d} + O(n^{d-1})\right)}\\
=&\lim_{n \to \infty} \frac{n^d}{\left(\frac{1}{d+1} - \frac{1}{2^d}\right)n^d + O(n^{d-1})}\\
=& \lim_{n \to \infty} \frac{1}{\left(\frac{1}{d+1} - \frac{1}{2^d}\right) + O\left(\frac{1}{n}\right)}\\
=&\frac{1}{\frac{1}{d+1} - \frac{1}{2^d}}\\
=& \frac{2^d(d+1)}{2^d-(d+1)},
\end{align*}
from Theorem \ref{th:est} we have
$$\lim_{n \to \infty} \widehat{\rho}_d^{\alpha} = \frac{2^d(d+1)}{2^d-(d+1)} \sum_{k=0}^{|J|} (-1)^k \frac{2^{|I|+k}-(|I|+k+1)}{2^{|I|+k}(|I|+k+1)} \sum_{\substack{S \subseteq J \\ |S|=k}} \rho_{X_{I\cup S}}^- = \rho_d^\alpha(C).$$
\end{remark}

\begin{remark}
Note that, for instance, the estimators for dimensions 3 and 4 are given by
$$\widehat{\rho}_3^{\alpha} = \frac{8(n+1)}{n-1} \left( \sum_{\substack{S \subseteq J \\ |I|+|S|=2}} \frac{1}{12} \widehat{\rho}_{X_{I\cup S}}^- - \sum_{\substack{S \subseteq J \\ |I|+|S|=3}} \frac{1}{8} \widehat{\rho}_{X_{I\cup S}}^- \right)$$
and
$$\widehat{\rho}_4^{\alpha} = \frac{240(n+1)^3}{33n^3+27n^2-37n-23} \left( \sum_{\substack{S \subseteq J \\ |I|+|S|=2}} \frac{1}{12} \widehat{\rho}_{X_{I\cup S}}^- - \sum_{\substack{S \subseteq J \\ |I|+|S|=3}} \frac{1}{8} \widehat{\rho}_{X_{I\cup S}}^- + \sum_{\substack{S \subseteq J \\ |I|+|S|=4}} \frac{11}{80} \widehat{\rho}_{X_{I\cup S}}^- \right),$$
respectively.
\end{remark}

\subsection{Empirical processes and some properties of these estimators}

\subsubsection{Empirical process}

In \cite{Gar2013}, an empirical process was introduced in order to estimate a generalization of a $3$-copula which can be easily extended to the $n$-dimensional case. As the authors did, we commence considering $\upbeta$ an index set, such that $\upbeta\subseteq\{1,\ldots,d\}$. We define $x_\upbeta=(x_{\upbeta_1},\ldots,x_{\upbeta_k})$ an arbitrary value of $\XX_\upbeta$. Let $|\upbeta|$ denote the cardinal of $\upbeta$ and $\alpha\in\mathbb{R}^d$ a direction. Consider the function 
    \begin{equation}\label{H}
     H_{\upbeta,\alpha}(x_\upbeta)=\mathbb{P}[\alpha_iX_i\leq\alpha_ix_i],\hspace{0.5cm}i\in\upbeta,  
    \end{equation}
i.e.,$$
H_{\upbeta,\alpha}(x_\upbeta) = \mathbb{P}\left[ \bigcap_{\substack{i \in \upbeta \\ \alpha_i=1}} \{X_i \leq x_i\}, \bigcap_{\substack{i \in \upbeta \\ \alpha_i=-1}} \{X_i \geq x_i\} \right].
$$

Let $u_\upbeta$ be the analogous of $x_\upbeta$ and $F_i$ the marginal cumulative distribution function of $X_i$. Then 
\begin{equation}\label{C}
C_{\upbeta,\alpha}(u_\upbeta) = H_{\upbeta,\alpha}((v_i)_{i \in \upbeta}),
\end{equation}
where
$$
v_k =
\begin{cases}
F_k^{-1}(u_i), & \text{if } \alpha_i = 1, \\
F_k^{-1}(1-u_i), & \text{if } \alpha_i = -1,
\end{cases}$$
is a generalization of the $d$-copula when $\alpha_i=1$ and $\upbeta=\{1,\ldots,d\}$. From this expression, we can introduce as in \cite{Gar2013} the empirical process to estimate it
\begin{equation}\label{estC}
C_{\upbeta,\alpha,n}(u_\upbeta)=\frac{1}{n+1}\sum_{j=1}^n\prod_{i\in\upbeta}{\bf 1}_{\left\{\alpha_i\frac{R_{ij}}{n+1}\leq\alpha_iu_i\right\}}.    
\end{equation}
When $\upbeta=\{1,\ldots,d\}$ and $\alpha={\bf 1}$, Expression (\ref{estC}) is the empirical estimator of the $d$-copula.

To determine the weak convergence ---denoted by ``$\xrightarrow[]{w}$''--- of the empirical process 
$$\sqrt{n}\{C_{\upbeta,\alpha,n}(u_\upbeta)-C_{\upbeta,\alpha}(u_\upbeta)\}, \hspace{0.5cm} u_{\upbeta}\in [0,1]^{|\upbeta|},$$
it suffices to proceed as in \cite{Gar2013}, where the following two conditions are considered:

\begin{condition}\label{con1}
    For each $i\in\upbeta$, the $i$-th first-order partial derivative $\dot{C}_{i,\upbeta,\alpha}$ exists and is continuous on the set $\{u_\upbeta\in[0,1]^{|\upbeta|}:0<u_i<1\}$.
\end{condition}

\begin{condition}\label{con2}
$\mathbb{B}_{C_{\upbeta,\alpha}}(u_\upbeta)$ is a $C_{\upbeta,\alpha}$-tight centered Gaussian process on $[0,1]^{|\upbeta|}$. $u_{\upbeta}^{(i)}$ is a vector such that its $j$-th component, for $j \in \upbeta$, is:
$$(u_{\upbeta}^{(i)})_j=\begin{cases}
1 & \text{if } j \neq i \text{ when } \alpha_j = 1, \\
0 & \text{if } j \neq i \text{ when } \alpha_j = -1, \\
u_i & \text{if } j = i.
\end{cases}$$
The covariance function is $\mathbb{E}(\mathbb{B}_{C_{\upbeta,\alpha}}(u_\upbeta),\mathbb{B}_{C_{\upbeta,\alpha}}(v_\upbeta))=C_{\upbeta,\alpha}(w_\upbeta)-C_{\upbeta,\alpha}(u_\upbeta)C_{\upbeta,\alpha}(v_\upbeta)$, where the $j$-th component $(w_\upbeta)_j$ is:
$$(w_\upbeta)_j=\begin{cases}
\min(u_j,v_j) & \text{if } \alpha_j = 1, \\
\max(u_j,v_j) & \text{if } \alpha_j = -1,
\end{cases}$$
for all $j \in \upbeta$.
\end{condition}

From Conditions \ref{con1} and \ref{con2}, the following result is obtained, which is valid for any dimension $d$ and, therefore, guarantees the weak convergence of our process.

\begin{theorem}\label{th:proc}
    Let $H_{\upbeta,\alpha}$ be a $|\upbeta|$-dimensional distribution function, given by (\ref{H}) with continuous marginal distributions $F_i$, $i\in\upbeta$ and with $C_{\upbeta,\alpha}$ given by (\ref{C}), where $\upbeta\subseteq\{1,...,d\}$ and $\alpha\in\mathbb{R}^d$, $\alpha_i=\{-1,1\}$. Under Condition \ref{con1} on the function $C_{\upbeta,\alpha}$, when $n\longrightarrow\infty$
    $$\sqrt{n}\left\{C_{\upbeta,\alpha,n}(u_\upbeta)-C_{\upbeta,\alpha}(u_\upbeta)\right\}\xrightarrow[]{w}\mathbb{G}_{C_{\upbeta,\alpha}}(u_\upbeta).$$
    Weak convergence takes place in $l^\infty([0,1]^{|\upbeta|})$ and $\mathbb{G}_{C_{\upbeta,\alpha}}(u_\upbeta)=\mathbb{B}_{C_{\upbeta,\alpha}}(u_\upbeta)-\sum_{i\in\upbeta}\dot{C}_{i,\upbeta,\alpha}(u_\upbeta)\mathbb{B}_{C_{\upbeta,\alpha}}(u_\upbeta^{(i)})$, where $\mathbb{B}_{C_{\upbeta,\alpha}}$ is a $C_{\upbeta,\alpha}$-tight centered Gaussian process on $[0,1]^{|\upbeta|}$ and $\mathbb{G}_{C_{\upbeta,\alpha}}$ follows Condition \ref{con2}.
\end{theorem}

\subsubsection{Properties of the estimators}\label{sect:prop}
Our main objective in this section is to prove that the estimators we have defined are unbiased. To do so, first, we must emphasize that our estimators verify the following:
\begin{equation}\label{eq:concl}
    \widehat{\rho}_d^\alpha=\frac{\frac{1}{n}(n+1)^d}{\frac{1}{n}\sum_{j=1}^nj^d-(\frac{n+1}{2})^d}\int_{\mathbb{I}^d}C_{\upbeta,\alpha,n}(u_\upbeta){\rm d}u_\upbeta-\frac{(\frac{n+1}{2})^d}{\frac{1}{n}\sum_{j=1}^nj^d-(\frac{n+1}{2})^d}, \hspace{0.5cm}\text{ for }\upbeta=\{1,\ldots,d\},
\end{equation}
since $$\int_{|\upbeta|}C_{\upbeta,\alpha,n}(u_\upbeta)du_\upbeta=\frac{1}{(n+1)^{|\upbeta|}}\sum_{j=1}^n\prod_{i\in\upbeta}R_{ij}^{\alpha_i}.$$

In order to prove that our estimators are unbiased we will use the following theorem from \cite{Gar2013}, which is an adaptation of \cite[Theorem 6]{Fer2004} and \cite{Gan87}.

\begin{theorem}\label{th:proc2}
    Under the assumptions of Theorem \ref{th:proc}, let  $\{a_n\}_{n\geq 1}$ and $\{b_n\}_{n\geq 1}$ be two sequences of real numbers satisfying $\sqrt{n}(a_n-a_0)=O(n^{-1/2})$ and $\sqrt{n}(b_n-b_0)=O(n^{-1/2})$, respectively, where $a_0$ and $b_0$ are constant values. Let $$T_n(f)=a_n\int_{I^{|\upbeta|}}f(u_\upbeta)du_\upbeta+b_n,$$
    for $n\geq 0$, where $f$ is a $|\upbeta|$-integrable function. Then, when $n\longrightarrow\infty$,
    $$\sqrt{n}\{T_n(C_{\upbeta,\alpha,n})-T_0(C_{\upbeta,\alpha})\}\xrightarrow{w}Z_{C_{\upbeta,\alpha}}\sim N(0,\sigma_{C_{\upbeta,\alpha}}^2),$$
    with $$\sigma_{C_{\upbeta,\alpha}}^2=a_0^2\int_{I^{|\upbeta|}}\int_{I^{|\upbeta|}}\mathbb{E}[\mathbb{G}_{C_{\upbeta,\alpha}}(u_\upbeta)\mathbb{G}_{C_{\upbeta,\alpha}}(v_\upbeta)]du\upbeta dv\upbeta$$ and $$Z_{C_{\upbeta,\alpha}}:=a_0\int_{I^{|\upbeta|}}\mathbb{G}_{C_{\upbeta,\alpha}}(u_\upbeta)du_\upbeta.$$
\end{theorem}

As a consequence, we have the following result.

\begin{corollary}\label{col:unb}
    Under assumptions of Theorem \ref{th:proc}, when $n\longrightarrow\infty$
    $$\sqrt{n}\{\widehat{\rho}_d^\alpha-\rho_d^\alpha\}\xrightarrow{w}Z_{C_{\upbeta,\alpha}}\sim N(0,\sigma_{C_{\upbeta,\alpha}}^2),$$
    where $\upbeta=\{1,2,\ldots,d\}$ and $\alpha\in\mathbb{R}^d$ is a vector such that $\alpha_i\in\{-1,1\}$ for all $i=1,\ldots,d$.
\end{corollary}

\begin{proof}
    We consider the real number sequences in Theorem \ref{th:proc2} as $$a_n=\frac{\frac{1}{n}(n+1)^d}{\frac{1}{n}\sum_{j=1}^nj^d-(\frac{n+1}{2})^d},\quad a_0=\frac{2^d(d+1)}{2^d-(d+1)},\quad b_n=\frac{(\frac{n+1}{2})^d}{\frac{1}{n}\sum_{j=1}^nj^d-(\frac{n+1}{2})^d}\quad\text{and}\quad b_0=\frac{d+1}{2^d-(d+1)}.$$
    We are going to check that the conditions $\sqrt{n}(a_n-a_0)=O(n^{-1/2})$ and $\sqrt{n}(b_n-b_0)=O(n^{-1/2})$ hold. We start by using the fact that
    $$\sum_{j=1}^nj^d=\frac{n^{d+1}}{d+1}+O(n^d).$$
    Therefore, after some tedious calculations, we obtain 
    \begin{align*}
        \sqrt{n}(a_n-a_0)=&\sqrt{n}\left(\frac{n^d+O(n^{d-1})}{n^d/(d+1)+O(n^{d-1})-\left(\frac{n+1}{2}\right)^d}-\frac{2^d(d+1)}{2^d-(d+1)}\right)\\
        &=\sqrt{n}\left(\frac{O(n^{d-1})}{O(n^d)}\right)\\
        &=O(n^{-1/2}).
    \end{align*}
    Analogously, 
    \begin{align*}
        \sqrt{n}(b_n-b_0)=&\sqrt{n}\left(\frac{(\frac{n+1}{2})^d}{n^d/(d+1)+O(n^{d-1})-\left(\frac{n+1}{2}\right)^d}-\frac{d+1}{2^d-(d+1)}\right)\\
        &=\sqrt{n}\left(\frac{O(n^{d-1})}{O(n^d)}\right)\\
        &=O(n^{-1/2}).
    \end{align*}
    So, using (\ref{eq:concl}), the desired result is obtained.
\end{proof}

By Corollary \ref{col:unb}, the estimators $\widehat{\rho}_d^\alpha$ are asymptotically unbiased and Var$(\widehat{\rho}_d^\alpha)\longrightarrow 0$ whenever $n\longrightarrow\infty$. Therefore, $\widehat{\rho}_d^\alpha$ is also an asymptotically consistent estimator since convergence in probability, $\widehat{\rho}_d^\alpha\xrightarrow{p}\rho_d^\alpha$, is guaranteed  as a consequence of Chebyshev's inequality.

\subsection{Simulations}

 To test the performance of these estimators, we have carried out Monte Carlo simulations for $d$-copulas of the parametric Clayton family, whose expression, with $\theta\in [0,\infty)$, is given by
 $$C_\theta(u_1,\ldots,u_d)=\left(\sum_{i=1}^d u_i^{-\theta}-d+1\right)^{-\frac{1}{\theta}}, \hspace{0.5cm}\forall \uu\in[0,1]^d$$
(see \cite{Ne06}). We consider two dimensions, $d=\{3,4\}$, and four sample sizes $n=\{20,50,100,500\}$. We have generated $1000$ Monte Carlo replicates of size $n$ for each parameter value and for each direction. In each of these cases, we have calculated the estimators, $\widehat{\rho}_d^\alpha$, of the directional $\rho$-coefficients, $\rho_d^\alpha$. Once obtained, we have calculated their mean, which appears in the tables presented. They also show an approximate value of the corresponding directional $\rho$-coefficient calculated by numerical integration. Table \ref{tab:1} displays the estimations of $\rho_3^{(-1,1,1)}$ for the $3$-copula of Clayton with sample sizes $n=\{20,50,100,500\}$ and parameter values $\theta=\{0.4, 0.6, 1, 2, 5\}$ comparing them with their exact value. Table \ref{tab:2} displays the estimations of $\rho_3^{(-1,-1,1)}$ for the $3$-copula of Clayton with sample sizes $n=\{20,50,100,500\}$ and parameter values $\theta=\{0.4, 0.6, 1, 2, 5\}$ comparing them with their exact value. Table \ref{tab:3} displays the estimations of $\rho_4^{(-1,1,1,-1)}$ for the $4$-copula of Clayton with sample sizes $n=\{20,50,100,500\}$ and parameter values $\theta=\{0.4, 0.6, 1, 2, 5\}$ comparing them with their exact value.

Likewise, the values of $\rho_4^{(-1,1,1-1)}$ have been estimated by the expression of Theorem \ref{th:est}. Table 4 shows the obtained results. Displayed values correspond to the mean of the estimators $\widehat{\rho}_{14}^-, \widehat{\rho}_{124}^-, \widehat{\rho}_{134}^-, \widehat{\rho}_{1234}^-$ and $\widehat{\rho}_4^{(-1,1,1,-1)}$ of $1000$ Monte Carlo simulations.

Analyzing the tables, we can see that these estimators are a good approximation to the real value of the directional $\rho$-coefficients and more so the larger the sample size considered. In all considered cases, negative values have been obtained for these coefficients. However, by changing the direction, for the Clayton $d$-copula, positive values can also be obtained, closer to one the larger the value of the parameter $\theta$ is, as would be expected. It can also be seen that the values obtained by estimating $\widehat{\rho}_4^{(-1,1,1,-1)}$ through the expression of Theorem \ref{th:est} in Table \ref{tab:4} provide values closer to the real value of the coefficient $\rho_4^{(-1,1,1,-1)}$, although the difference with the values estimated in Table \ref{tab:3} is not directly significant.

 \begin{table}[t]
\begin{center}
\begin{tabular}{ c  c  c  c  c c c  }
\hline
\multicolumn{3}{ c }{Sample size} & $n=20$  & $n=50$ & $n=100$ & $n=500$ \\ \hline
$\theta$ & $\rho_d^{\alpha}$ & $\widehat{\rho}_d^{\alpha}$ & mean & mean & mean & mean \\ \hline
0.4 & -0.0726 & $\widehat{\rho}_d^{(-1,1,1)}$ & -0.0740 & -0.0697 & -0.0735 & -0.0714\\
0.6 & -0.0969 &  $\widehat{\rho}_d^{(-1,1,1)}$ & -0.0924 & -0.0943 & -0.0966 &  -0.0969\\
1 & -0.1338  &  $\widehat{\rho}_d^{(-1,1,1)}$ &  -0.1313 &-0.1327 & -0.1355 & -0.1347\\
2 &  -0.1906 & $\widehat{\rho}_d^{(-1,1,1)}$ & -0.1895 & -0.1920 & -0.1917 & -0.1928 \\
5 & -0.2684 &  $\widehat{\rho}_d^{(-1,1,1)}$ & -0.2609 & -0.2661 & -0.2673 & -0.2686\\ \hline
\end{tabular}
\caption{Results obtained in Monte Carlo simulations in order to calculate the value of $\widehat{\rho}_d^{(-1,1,1)}$ based on a total of $1000$ samples of size $n$ generated for the Clayton $3$-copula for different values of the parameter $\theta$.} 
\label{tab:1}
\end{center}
\end{table}

 \begin{table}[t]
\begin{center}
\begin{tabular}{ c  c  c  c  c c c  }
\hline
\multicolumn{3}{ c }{Sample size} & $n=20$  & $n=50$ & $n=100$ & $n=500$ \\ \hline
$\theta$ & $\rho_d^{\alpha}$ & $\widehat{\rho}_d^{\alpha}$ & mean & mean & mean & mean \\ \hline
0.4 & -0.0919 & $\widehat{\rho}_d^{(-1,-1,1)}$ & -0.0848 & -0.0906 & -0.0918  & -0.0919 \\
0.6 & -0.1287 &  $\widehat{\rho}_d^{(-1,-1,1)}$ & -0.1203  & -0.1284 & -0.1277  & -0.1278  \\
1 &  -0.1850 &  $\widehat{\rho}_d^{(-1,-1,1)}$ & -0.1778  & -0.1826 & -0.1827 & -0.1849 \\
2 & -0.2621  & $\widehat{\rho}_d^{(-1,-1,1)}$ & -0.2478 & -0.2580 & -0.2591 & -0.2612 \\
5 & -0.3212 &  $\widehat{\rho}_d^{(-1,-1,1)}$ & -0.3105 & -0.3162  & -0.3195 & -0.3207\\ \hline
\end{tabular}
\caption{Results obtained in Monte Carlo simulations in order to calculate the value of $\widehat{\rho}_d^{(-1,-1,1)}$ based on a total of $1000$ samples of size $n$ generated for the Clayton $3$-copula for different values of the parameter $\theta$.}
\label{tab:2}
\end{center}
\end{table}

 \begin{table}[h!]
\begin{center}
\begin{tabular}{ c  c  c  c  c c c  }
\hline
\multicolumn{3}{ c }{Sample size} & $n=20$  & $n=50$ & $n=100$ & $n=500$ \\ \hline
$\theta$ & $\rho_d^{\alpha}$ & $\widehat{\rho}_d^{\alpha}$ & mean & mean & mean & mean \\ \hline
0.4 & -0.0664 & $\widehat{\rho}_d^{(-1,1,1,-1)}$ & -0.0650 & -0.0661 & -0.0652  & -0.0670 \\
0.6 & -0.0876 &  $\widehat{\rho}_d^{(-1,1,1,-1)}$ & -0.0888  & -0.0880 & -0.0866  & -0.0874  \\
1 & -0.1176  &  $\widehat{\rho}_d^{(-1,1,1,-1)}$ & -0.1165 & -0.1180 & -0.1189 & -0.1169 \\
2 & -0.1583  & $\widehat{\rho}_d^{(-1,1,1,-1)}$ & -0.1612  & -0.1594 & -0.1590 & -0.1584 \\
5 & -0.1966 &  $\widehat{\rho}_d^{(-1,1,1,-1)}$ & -0.2017 &  -0.1985 & -0.1975 & -0.1967 \\ \hline
\end{tabular}
\caption{Results obtained in Monte Carlo simulations in order to calculate the value of $\widehat{\rho}_d^{(-1,1,1,-1)}$ based on a total of $1000$ samples of size $n$ generated for the Clayton $4$-copula for different values of the parameter $\theta$.}
\label{tab:3}
\end{center}
\end{table}

 \begin{table}[t]
\begin{center}
\begin{tabular}{ l c  c  c c c c c }
\hline
Sample size &$\theta$ & $\rho_4^{(-1,1,1-1)}$ & $\widehat{\rho}_{14}^-$ & $\widehat{\rho}_{124}^-$& $\widehat{\rho}_{134}^-$& $\widehat{\rho}_{1234}^-$ & $\widehat{\rho}_{4}^{(-1,1,1,-1)}$  \\ \hline
$n=20$& \multirow {4}{*}{0.4} & \multirow{4}{*}{-0.0664} & 0.2258 & 0.2244 & 0.2247 & 0.2068  & -0.0650 \\
$n=50$&  &  & 0.2357 & 0.2336 & 0.2294 & 0.2122  & -0.0658 \\
$n=100$&  &  & 0.2456 & 0.2393 & 0.2364 & 0.2163  & -0.0672 \\
$n=200$&  &  & 0.2471 & 0.2388 & 0.2375 & 0.2169  & -0.0664 \\
$n=20$ & \multirow{4}{*}{0.6} & \multirow{4}{*}{-0.0876} &  0.3220 & 0.31047  & 0.3062 & 0.2827  & -0.0828  \\
$n=50$ &  &  & 0.3376  & 0.3241 & 0.3190 & 0.2939  & -0.0862  \\
$n=100$ &  &  & 0.3373  & 0.3209  & 0.3199 & 0.2922  & -0.0860  \\
$n=200$ &  &  & 0.3376  & 0.3210  & 0.3223 & 0.2924  & -0.0878  \\
$n=20$ & \multirow{4}{*}{1} & \multirow{4}{*}{-0.1176}  &  0.4499 & 0.4312 & 0.4318 & 0.4006 & -0.1113 \\
$n=50$ &  &   &  0.4645 & 0.4441 & 0.4448 & 0.4116 & -0.1149 \\
$n=100$ &  &   & 0.4754  & 0.4510 & 0.4511 & 0.4140 & -0.1179 \\
$n=200$ &  &   & 0.4765 & 0.4520 & 0.4522 & 0.4156 & -0.1176 \\
$n=20$ & \multirow{4}{*}{2} & \multirow{4}{*}{-0.1583}  & 0.6593 & 0.6309  & 0.6312 & 0.5934 & -0.1545 \\
$n=50$ &  &   & 0.6680 & 0.6374  & 0.6350 & 0.5949 & -0.1570 \\
$n=100$ &  &   & 0.6743 & 0.6412  & 0.6418 & 0.6000 & -0.1582 \\
$n=200$ &  &   & 0.6773 & 0.6439  & 0.6443 & 0.6022 & -0.1583 \\
$n=20$ & \multirow{4}{*}{5} & \multirow{4}{*}{-0.1966} & 0.8630 & 0.8360 & 0.8368 & 0.8054 & -0.1923 \\ 
$n=50$ & & & 0.8732 & 0.8470 & 0.8473 & 0.8167 & -0.1944 \\ 
$n=100$ & & & 0.8792 & 0.8535 & 0.8534 & 0.8234 & -0.1954 \\ 
$n=200$ & & & 0.8816 & 0.8558 & 0.8562 & 0.8256 & -0.1965 \\ 
\hline
\end{tabular}
\caption{Results obtained in Monte Carlo simulations using the expression in Theorem \ref{th:est} in order to calculate the value of $\widehat{\rho}_d^{(-1,1,1,-1)}$ based on a total of $1000$ samples of size $n$ generated for the Clayton $4$-copula for different values of the parameter $\theta$.}
\label{tab:4}
\end{center}
\end{table}

\section{Applications}\label{sec:appl}
In this final section, we explore potential applications of the proposed directional dependence estimators in various fields, with a particular focus on health, climate, and rainfall studies. By leveraging these estimators, researchers can gain deeper insights into the directional dependencies that exist between different components of multivariate data, which are often overlooked in traditional analyses. In the context of health, for example, understanding the directional relationships between risk factors and health outcomes can improve predictive models. Similarly, in climate and rainfall studies, these estimators can be instrumental in analyzing how different atmospheric variables influence each other in a directional manner, thereby enhancing forecasting and decision-making processes. Through these applications, we highlight the practical value of directional dependence measures in addressing complex, real-world problems particularly relevant in finance and economics, as the following example illustrates.

\begin{example}
     In this example we will model the association of three companies: Intel (INTC), General Electrics (GE) and Microsoft (MSFT). To do this, we are going to use a database of daily log-returns of the three companies for five years (1996--2000) obtained from the QRM (Quantitative Risk Management) package of R \cite{Pfaff}. In Figure \ref{fig emp} we can see the three-dimensional scatterplot of the data available for the three companies, as well as two-dimensional scatterplots of the data involving only two companies.

      \begin{figure}[h!]
\centering
\begin{subfigure}[b]{0.4\linewidth}
\includegraphics[width=\linewidth]{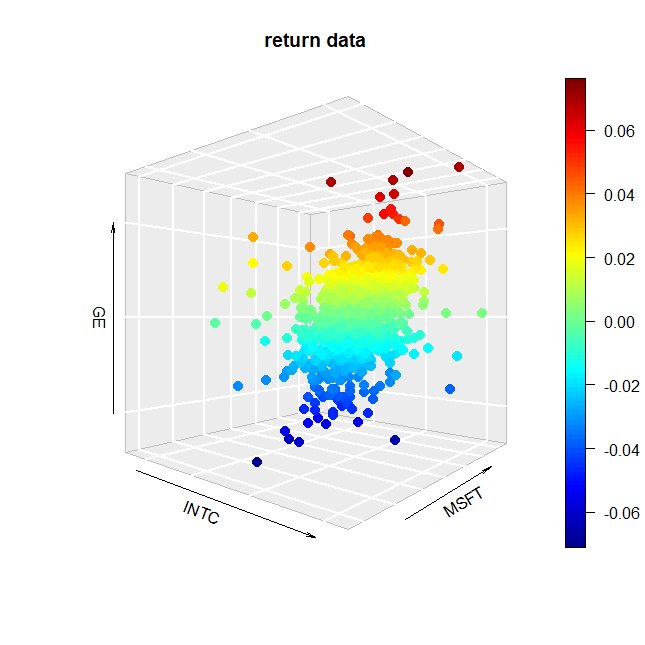}
\caption{$3$-dimensional scatterplot.}
\label{}
\end{subfigure}
\begin{subfigure}[b]{0.5\linewidth}
\includegraphics[width=\linewidth]{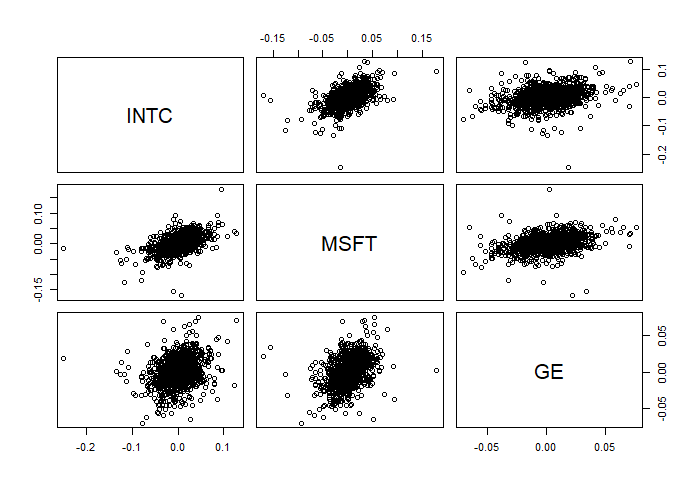}
\caption{$2$-dimensional scatterplots.}
\label{}
\end{subfigure}
\caption{Scatterplots of log-return data.}
\label{fig emp}
\end{figure}
To study the levels of directional association of these companies, the directional $\rho$-coefficients have been calculated for each of the possible directions of $\mathbb{R}^3$, obtaining Table \ref{tab:data}.

\begin{table}[h!]
\begin{center}
\begin{tabular}{ c  c  }
\hline
$\alpha$ & $\widehat{\rho}_3^\alpha$ \\ \hline
(-1,-1,-1) & 0.4400  \\
(1,1,1) & 0.4330   \\
(-1,1,1) & -0.1640\\
(1,-1,1) & -0.2060   \\
(1,1,-1) & -0.0525 \\ 
(-1-1,1) & -0.0605\\
(-1,1,-1) & -0.2160\\
(1,-1,-1) & -0.1741\\ \hline
\end{tabular}
\caption{Different values obtained for the directional estimators according to the data.}
\label{tab:data}
\end{center}
\end{table}

The largest values for the estimators are in the directions $(1,1,1)$ and $(-1,-1,-1)$, which can be interpreted as ``large" values of all three variables tend to occur simultaneously, and the same with ``small" values.
\end{example}

\section{Conclusion}\label{sec:conc}
In this paper, the authors have explored the use of copulas and directional $\rho$-coefficients as effective tools for measuring directional association between random variables. We have presented the generalization of directional $\rho$-coefficients to higher dimensions, highlighting their key properties and theoretical implications. Additionally, we addressed the previously posed conjecture in a more generalized manner, providing new insights into the behavior of these coefficients in multivariate settings. Nonparametric rank-based estimators were also introduced as a means to estimate the values of the directional $\rho$-coefficients from observed data, offering a robust and flexible approach to capturing dependency patterns in empirical datasets.

Studying directional dependence offers several advantages, particularly in capturing the directional nature of relationships between random variables. Unlike traditional measures of association, which often treat dependency symmetrically, directional coefficients allow us to distinguish how the relationship between random variables may differ in magnitude depending on the direction of change. This may be especially valuable in fields such as finance, economics, and environmental science, where understanding the direction of dependency can provide more accurate predictions and decision-making tools.

In conclusion, the exploration of directional dependence through copulas and $\rho$-coefficients offers valuable insights into the relationships between variables, with a wide range of potential applications across various scientific disciplines. Future work in this area promises to expand our understanding of complex dependencies and provide new tools for analyzing high-dimensional data.

\bigskip

\noindent
{\bf Acknowledgments}
\medskip

The first and third authors acknowledge the support of the research project PID2021-122657OB-I00
by the Ministry of Science, Innovation and Universities (Spain). The first author thanks to P\_FORT\_GRU-\\ POS\_2023/76, PPIT-UAL, Junta de Andaluc\'ia-ERDF 2021-2027. Programme: 54.A, and is also partially supported by the CDTIME of the University of Almería (Spain). Third author also acknowledges the support of P\_FORT\_GRUPOS\_2023/104, PPIT-UAL, Junta de Andaluc\'ia-ERDF 2021-2027, Objective RS01.1, Programme 54.A.

\end{document}